\documentclass[11pt]{amsart}
\usepackage{amssymb, amsmath, latexsym,   bm, mathrsfs}
\usepackage{multirow, xcolor, float}

\usepackage[shortlabels]{enumitem}
\usepackage{tasks}
\usepackage[numbers,sort&compress]{natbib}
\usepackage{soul}
\sodef\lightspacing{}{.09em}{.2em plus .2em}{.5em plus .1em minus .1em}

\renewcommand{\baselinestretch}{\baselinestretch}
\renewcommand{\baselinestretch}{1.1}
\numberwithin{equation}{section}

\newtheorem{thm}{Theorem}[section]

\newtheorem{prop}[thm]{Proposition}

\newcommand{\Mod}[1]{\ (\mathrm{mod}\ #1)}

\newcommand{\sMod}[1]{\ (\mathrm{mod}\ #1)}

\begin{document}

\title[Sums of squares of integers except for a fixed one]{Sums of squares of integers except for a fixed one}

\author[Chae et al.]{Wonjun Chae}
\address{Department of Mathematical Sciences and Research Institute of Mathematics, Seoul National University, Seoul 08826, Korea}
\email{wonjun97@snu.ac.kr}
\thanks{This work of the first and the fifth authors was supported by the National Research Foundation of Korea(NRF) grant funded by the Korea government(MSIT)(NRF-2020R1A5A1016126) and (RS-2024-00342122).}

\author[]{Yun-Seong Ji}
\address{Research Institute of Mathematics, Seoul National University, Seoul 08826, Korea}
\email{ys0430@snu.ac.kr}
\thanks{This work of the second author was supported by the National Research Foundation of Korea(NRF) grant funded by the Korea government(MSIT)(NRF-2021R1I1A1A01043959) and (RS-2024-00342122).}

\author[]{Kisuk Kim}
\address{Department of Mathematical Sciences and Research Institute of Mathematics, Seoul National University, Seoul 08826, Korea}
\email{samsungkks@snu.ac.kr}

\author[]{Kyoungmin Kim}
\address{Department of Mathematics, Hannam University, Daejeon 34430, korea}
\email{kiny30@hnu.kr}

\author[]{Byeong-Kweon Oh}
\address{Department of Mathematical Sciences and Research Institute of Mathematics, Seoul National University, Seoul 08826, Korea}
\email{bkoh@snu.ac.kr}

\author[]{Jongheun Yoon}
\address{Charles University, Faculty of Mathematics and Physics, Department of Algebra, Sokolovská 83, 186 00 Praha 8, Czech Republic}
\email{jongheun.yoon@matfyz.cuni.cz}
\thanks{This work of the sixth author was supported by the National Research Foundation of Korea(NRF) grant funded by the Korea government(MSIT)(RS-2024-00342122) and grant 21-00420M from Czech Science Foundation (GA\v{C}R)}

\subjclass[2020]{Primary 11E20, 11E25}

\keywords{sums of squares}

\begin{abstract}
In this article, we study a sum of squares of integers except for a fixed one. For any nonnegative integer $n$, we find the minimum number of squares  of integers except for $n$ whose sums represent all positive integers that are represented by a sum of squares except for it. This problem could be considered as a generalization of Dubouis's result \cite{Du1911} for the case when $n=0$.
\end{abstract}

\maketitle


\section{Introduction} \label{sec:intro}



In 1770, Lagrange proved in \cite{La} that any positive integer can be represented by a sum of at most four squares of integers. That is, for the quaternary quadratic form $f(x,y,z,w)=x^2+y^2+z^2+w^2$, the Diophantine equation 
$$
f(x,y,z,w)=x^2+y^2+z^2+w^2=n
$$
 has an integer solution for any positive integer $n$.
This famous Lagrange's four square theorem has been generalized in many directions. Ramanujan determined in \cite{R} that there are $55$ positive definite integral diagonal quadratic forms which represent all nonnegative integers. Such a quadratic form is called {\it universal}.
Dickson later pointed out in \cite{Di} that the form $x^2+2y^2+5z^2+5w^2$, which represents all positive integers except for $15$, was incorrectly included, leaving only $54$. Recently, Conway and Schneeberger proved, so called, the $15$-theorem, which says that any positive definite integral quadratic form representing $1,2,3,5,6,7,10,14$, and $15$ represents all nonnegative integers. Bhargava provided in \cite{Bh} a simple and elegant method to prove
 the $15$-theorem.

In 1911, Dubouis determined in \cite{Du1911}   all positive integers that are not represented by sums of $k$ \textit{nonzero} squares for any $k \geq 4$. Following by Dubouis's result, any positive integer is represented by $k$ nonzero integers except for the integers in $E_k$, where
\begin{align*}
     E_4&=\{1,2,3,5,6,8,9,11,14,17,29,41\} \cup \{2\cdot4^m, 6\cdot 4^m, 14\cdot 4^m \mid m \geq0\},\\ 
    E_5&=\{1,2,3,4,6,7,9,10,12,15,18,33\},
    \intertext{and}
    E_k&=\{1,2,\dots, k-1, k+1, k+2, k+4, k+5, k+7, k+10, k+13 \}  
\end{align*} 
for any $k \geq 6$.
 As a generalization, Ji, Kim,
 and Oh determined in \cite{jko2015} all positive definite binary quadratic forms that are represented by  sums of $k$ nonvanishing squares for any integer $k\ge 5$.  In \cite{KK2017}, Kim and Kim extend the Dubouis's results to the real quadratic field $\mathbb{Q}(\sqrt{m})$. In fact, they proved that for any integer $k \geq 5$, there exists a bound $C(m,k)$ such that any totally positive integer in the real quadratic field $\mathbb{Q}(\sqrt{m})$ whose norm exceeds $C(m,k)$ can be represented as a sum of $k$ nonzero integral squares in $\mathbb{Q}(\sqrt{m})$.  In \cite{KO2022}, Kim and Oh extend the Dubouis's results in another way. They determined $S(p)$ for a prime $p$ that is the smallest number $k$ such that any positive integer is a sum of at most $k$ squares of integers that are not divisible by $p$.

In this article, we extend Dubouis's result in a different direction. To explain our result, we introduce some useful notation.
 For any positive integer $\rho$, we define 
$$
S_{\rho}=\{ n \in \mathbb Z \mid n\ge 0 \ \text{and} \ n\ne \rho\}.
$$
For any positive integer $n$, we define 
\[
    \Sigma_\rho(n)=\{m \mid \exists x_1,x_2,\dots,x_m \in S_{\rho} \ 
 \text{such that} \ n=x_1^2+\cdots+x_m^2\}. 
\]
If $\Sigma_\rho(n)$ is nonempty, we define
\[
k_{\rho}(n)=\min\Sigma_\rho(n).      
\]
When $\Sigma_\rho(n)$ is empty for given $\rho$ and $n$, we define $k_{\rho}(n)=\infty$. We also define
\[
    I(\rho)=\{n \mid k_{\rho}(n)=\infty\} \text{ and } M(\rho)=\max\{k_{\rho}(n) \mid n\not\in I(\rho) \}.
\]

The aim of this article is to determine $M(\rho)$ for any positive integer $\rho$. More precisely, we prove that\[
 M(\rho) = \begin{cases}
\ 8 & \text{if \ $\rho=2$,}\\
\ 6 & \text{if \ $\rho=1$, $3$,}\\
\ 5 & \text{if \ $\rho=5,\ 2^{m+1},\ 3\cdot2^m $ for some positive integer $m$,}\\
\ 4 & \text{otherwise.}\\
 \end{cases}
\]
Furthermore, we determine all positive  integers $\rho$ and $n$ such that $k_{\rho}(n) \ge 5$.


\section{Large integers that are sums of 4 squares of integers except for a fixed one}


In this section, we show in Theorem~\ref{550} that for any positive integer $\rho$, any sufficiently large integer can be represented by a sum of at most four squares of integers in $S_\rho$. Also, by using Theorem~\ref{550}, we determine the values of $M(\rho)$ when $1 \le \rho \le 5$.

One may easily show that there are exactly $12$ positive integers which are not a sum of squares of integers greater than $1$. In fact, they are
\[I(1)=\{1,2,3,5,6,7,10,11,14,15,19,23\}\text.\]
For any $\rho\geq2$, since $1 \in S_{\rho}$ and 
$$
n=\overbrace{1^2+1^2+\cdots+1^2}^{n-\text{times}},
$$
 for any positive integer $n$, we have  $I(\rho)=\varnothing$.

\begin{thm}\label{550} Let $\rho$ be any positive integer. Any integer $n$ with $n \geq 550\rho^{2}$ is a sum of at most four squares of integers in $S_{\rho}$.
\end{thm}

\begin{proof} Let $n$ be an integer satisfying the above inequality. 
    First, suppose that $n\equiv2\Mod4$. Choose a nonnegative integer $m$ such that
    \[
        4m^2< n<4(m+1)^2.
    \]
    Since $n-4m^2\equiv 2\Mod4$, it is represented by the ternary quadratic form $x^2+y^2+z^2+(x+y+z)^2$. So there are integers $\alpha,\beta,\gamma$, and $\delta$ such that
    \[\left\{\begin{aligned}
        &\ \alpha^2+\beta^2+\gamma^2+\delta^2=n-4m^2,\\
        &\ \alpha+\beta+\gamma+\delta=0.
    \end{aligned}\right.\]
    Therefore we have
    \[
        n=(m+\alpha)^2+(m+\beta)^2+(m+\gamma)^2+(m+\delta)^2.
    \]
    Note that $n-4m^2<8m+4$. Hence, we have
    \[
        -\sqrt{8m+4}<\alpha,\beta,\gamma,\delta<\sqrt{8m+4}\text.
    \] 
    This implies that $n=x^2+y^2+z^2+w^2$ has an integer solution $(x,y,z,w)=(a,b,c,d)$ such that 
    $$
    m-\sqrt{8m+4}<a,b,c,d<m+\sqrt{8m+4}.
    $$
    Hence, if $n$ is sufficiently large so that $m-\sqrt{8m+4}>\rho$, then all of $a,b,c,$ and $d$ are greater than $\rho$. Note that $m-\sqrt{8m+4}>\rho$ if and only if  $m>\rho+4+2\sqrt{2\rho+5}$. Assume that an integer $n$ satisfies
    \begin{align}
    \tag{$\ast$} n>4(\rho+5+2\sqrt{2\rho+5})^2=4\rho^2+72\rho+(16\rho+80)\sqrt{2\rho+5}+180.
    \end{align}
    Since
    $$
    4(m+1)^2>n>4(\rho+4+2\sqrt{2\rho+5}+1)^2,
    $$
   we have $m>\rho+4+2\sqrt{2\rho+5}$. Therefore if an integer $n$ satisfies the condition ($\textasteriskcentered$), 
    then we may find integers $a,b,c$, and $d$ such that 
$$
n=a^2+b^2+c^2+d^2\ \text{and} \ a,b,c,d \in S_{\rho}.
$$
    Since as a function of the variable $\rho$, 
    $$
    \frac{4(\rho+5+2\sqrt{2\rho+5})^2}{\rho^2}
    $$
    is decreasing if $\rho>0$, we have
    \[
        550 >4(6+2\sqrt7)^2 \ge \frac{4(\rho+5+2\sqrt{2\rho+5})^2}{\rho^2},
    \]
for any positive integer $\rho$. Therefore any integer $n$ with  $n \ge 550\rho^2$ satisfies the condition ($\textasteriskcentered$).

    Now, assume that $n\equiv1\Mod2$. First, choose a nonnegative integer $m$ such that \begin{equation*}
        4m^2+2m< n<4(m+1)^2+2(m+1)\text{.}
    \end{equation*}
    Consider the system of Diophantine equations
    \[\left\{\begin{aligned}
    &\ p^2+q^2+r^2+s^2=n-(4m^2+2m),\\
    &\ p+q+r+s=1.
    \end{aligned}\right.\]
Since $1<4(n-4m^2-2m)$, we may apply Cauchy's lemma in \cite{Ca1815}(see also \cite{N}) to obtain an integer solution  $(p,q,r,s)=(\alpha,\beta,\gamma,\delta)$ to the above system of equations. Then we have
\begin{align*}
     n&=4m^2+2m+\alpha^2+\beta^2+\gamma^2+\delta^2\\&=(m+\alpha)^2+(m+\beta)^2+(m+\gamma)^2+(m+\delta)^2.
\end{align*} 
Note that
\[
    n-4m^2-2m<4(m+1)^2+2(m+1)-4m^2-2m=8m+6\text,
\] which implies that
\[
    -\sqrt{8m+6}<\alpha,\beta,\gamma,\delta<\sqrt{8m+6}\text.
\]
Therefore the equation $n=x^2+y^2+z^2+w^2$ has an integer solution $(x,y,z,w)=(a,b,c,d)$ such that 
$$
m-\sqrt{8m+6}<a,b,c,d<m+\sqrt{8m+6}.
$$
Hence, if $n$ is sufficiently large so that $m-\sqrt{8m+6}>\rho$, then all of $a,b,c,$ and $d$ are greater than $\rho$. Note that $m-\sqrt{8m+6}>\rho$ if and only if $m>\rho+4+\sqrt{8\rho+22}$. Thus, for any integer $n$ such that
\begin{equation}
   \tag{$\ast\ast$}
\begin{aligned}
    n&> 4(\rho+5+\sqrt{8\rho+22})^2+2(\rho+5+\sqrt{8\rho+22})\\
    &=4\rho^2+74\rho+(8\rho+42)\sqrt{8\rho+22}+198,
\end{aligned}
\end{equation}
we may find integers $a,b,c,$ and $d$ such that
$$
n=a^2+b^2+c^2+d^2\ \text{and} \ a,b,c,d \in S_{\rho}.
$$
Since as a function of the variable $\rho$, 
$$
\frac{4(\rho+5+\sqrt{8\rho+22})^2+2(\rho+5+\sqrt{8\rho+22})}{\rho^2}
$$
is decreasing if $\rho>0$, we have
$$
550>4(6+\sqrt{30})^2+2(6+\sqrt{30})\geq \frac{4(\rho\!+\!5\!+\!\sqrt{8\rho+22})^2+2(\rho\!+\!5\!+\!\sqrt{8\rho+22})}{\rho^2},
$$
for any positive integer $\rho$.
 Therefore any integer $n$ with  $n \ge 550\rho^2$ satisfies the condition ($\textasteriskcentered\textasteriskcentered$).

Finally, assume that $n\equiv0\Mod4$. Denote $\rho=2^k\rho'$, where $k\geq0$ and $\rho'$ is an odd integer. Let $n=4^el$, where $e\geq1$ and $l\not\equiv0\Mod4$. By Lagrange's four square theorem, there are integers $a,b,c,$ and $d$ such that $l=a^2+b^2+c^2+d^2$.
Hence, we have \begin{equation*}
    n=(2^ea)^2+(2^eb)^2+(2^ec)^2+(2^ed)^2.
\end{equation*}
If $e>k$, then all of $2^ea,2^eb,2^ec,$ and $2^ed$ are not equal to $\rho$.

Suppose that $e \leq k$. If $l\ge 550(2^{k-e}\rho')^2$, then by the above arguments for any integer not divisible by $4$, we may find integers   $a',b',c'$, and $d'$   such that 
$$
l=(a')^2+(b')^2+(c')^2+(d')^2 \ \text{and} \ a',b',c',d' \in S_{2^{k-e}\rho'}.
$$ 
 Hence, if $n \ge 550\rho^2$, then we have \begin{equation*}
    n=4^el=(2^ea')^2+(2^eb')^2+(2^ec')^2+(2^ed')^2 \ \text{and} \ 2^ea', 2^eb', 2^ec', 2^ed' \in S_{\rho}.
\end{equation*} 
 The theorem follows directly from this.  
\end{proof}

\begin{thm} Any integer which is a sum of squares of integers greater than $1$ is a sum of at most six squares of integers greater than $1$ and hence  $M(1)=6$.
\end{thm}

\begin{proof}
By Theorem~\ref{550}, any integer greater than or equal to $550$ is a sum of at most four squares of integers greater than $1$. For any integer less than $550$, one may directly compute that \[
 k_1(n) = \begin{cases}
\ \infty & \text{if \ $n=1,2,3,5,6,7,10,11,14,15,19,23$,}\\
\ 6 & \text{if \ $n=39,55$,}\\
\ 5 & \text{if \ $n=30,35,46,51$,}\\
 \end{cases}
\]
and $k_1(n)\leq 4$ for all the other integers not given in the above. This completes the proof.
\end{proof}

In the same way, one may easily prove by using Theorem~\ref{550} that 
\begin{equation*}
    M(2) = 8\text{, } M(3) = 6\text{, } M(4) = 5  \text{, and } M(5) = 5\text.
\end{equation*}
Furthermore, we have
\[
 k_2(n) = \begin{cases} 
\  8 & \text{if \ $n=8,24$,}\\
\  7 & \text{if \ $n=7,15,23,31$,}\\
\  6 & \text{if \ $n=6,14,22,30$,}\\
\  5 & \text{if \ $n=5,13,21,29,40,56,120,184,$}\\
 \end{cases}
\]
and $k_2(n)\leq 4$ for all the other integers not given in the above.
For the case when $\rho=3$, we have
\[
k_3(n) = \begin{cases}
\  6 & \text{if \ $n=15$,}\\
\  5 & \text{if \ $n=11,14,23,35,47,59,71,95$,}\\
 \end{cases}
\]
and $k_3(n)\leq 4$ for all the other integers not given in the above.
Furthermore,
\[
k_4(n) = 5 \quad \text{if \ $ n=24,32,56,88,96,120,160,224,480,736,$}
\]
and $k_4(n)\leq 4$ for all the other integers not given in the above.
Finally, $k_5(79) =5$ and $k_5(n)\leq 4$ for any integer $n\neq 79$.


\section{Complete determination of $M(\rho)$}

Throughout this section, we always assume that $\rho\geq 6$. We prove that $M(\rho)=4$ if $\rho$ is divisible by $9$ or if it has a prime divisor greater than $3$, and $M(\rho)=5$ otherwise.

\begin{thm}
    If $\rho$ has a prime divisor greater than $3$, then $M(\rho)=4$.
\end{thm}

\begin{proof}
 In \cite{KO2022},  it was shown that any integer $n$ is a sum of at most $4$ squares of integers not divisible by a prime $p$ for any $p\geq 5$, except for the case when $p=5$ and $n=79$.  In the exceptional case,
 since
 \begin{equation*}
    79=1^2+2^2+5^2+7^2=2^2+5^2+5^2+5^2=3^2+3^2+5^2+6^2,
\end{equation*}
and we are assuming that $\rho \ge 6$, we have $M(\rho)=4$.
\end{proof}

Now, we consider the case where $\rho$ has no prime divisor greater than $3$.

\begin{thm}\label{3.2} Let $a$ be an integer greater than or equal to $2$ and let $\rho=2^{a+1}$. Then any positive integer $n$ except for the integers in 
\[
N(\rho)=\left\lbrace 2\rho^2, 6\rho^2, 10\rho^2, 14\rho^2, 30\rho^2, 46\rho^2,6\left(\frac{\rho}{2}\right)^2\!\!,  14\left(\frac{\rho}{2}\right)^2\!\!,  22\left(\frac{\rho}{2}\right)^2\!\!, 30\left(\frac{\rho}{2}\right)^2 \right\rbrace
\]
is a sum of $4$ squares of integers in $S_{\rho}$. Furthermore, any integer in $N(\rho)$ is a sum of $5$ squares of integers in $S_{\rho}$ and hence $M(\rho)=5$.
\end{thm}

\begin{proof} If the integer $a$ is less than or equal to $6$,  then one may prove the theorem by using  Theorem~\ref{550}.  Hence we always assume that $a\ge 7$.
By Lagrange's four square theorem, we may further assume that $n \ge 2^{16}$.

 First, assume that $n \equiv 1 \Mod 4$. Since the other case can be done in a similar manner, we further assume that $n \equiv 1 \Mod 8$. Assume that $a$ is odd. 
Since $n-2^2$ is a sum of three squares of integers, there are nonnegative integers $x,y,z$  with $x \equiv 1 \Mod 2, \ y \equiv 2 \Mod 4$, and $z\equiv 0 \Mod 4$ such that
$$
n-2^2=x^2+y^2+z^2.
$$ 
If $z \ne \rho$, then $n$ is a sum of $4$ squares in $S_{\rho}$. If $z=\rho$, then 
$$
n=x^2+y^2+2^{2a+2}+4=x^2+y^2+(2^{a+1}-2)^2+(2^{\frac{a+3}2})^2,
$$ 
which is desired. 

Now, assume that $a$ is even. Similarly to the above, there are nonnegative integers $x,y,z$ with 
$$
x \equiv 1 \Mod 2, \ y \equiv z \equiv 2 \Mod 4, \  \text{or}  \ x \equiv 1 \Mod 2, \ y \equiv z \equiv 0 \Mod 4
$$
such that
$$
n-4^2=x^2+y^2+z^2.
$$ 
If neither $y$ nor $z$  is equal to $\rho$, then $n$ is a sum of $4$ squares in $S_{\rho}$. If exactly one of $y$ and $z$, say $z$,  is equal to $\rho$, then 
$$
n=x^2+y^2+2^{2a+2}+4^2=x^2+y^2+(2^{a+1}-4)^2+(2^{\frac{a+4}2})^2,
$$ 
which is desired. Assume that both $y$ and $z$ are equal to $\rho$, that is, $n=x^2+2^{2a+3}+4^2$. By the similar reasoning to the above, we may assume that there is an integer $x_1$ such that $n=x_1^2+2^{2a+3}+16^2$.  By letting $a=2b$, note that 
$$
n=\left\{\begin{aligned}
  &\ x^2+(2^{2b+1}-4)^2+2^{2b+4}(2^{2b-2}+1), \\
  &\ x_1^2+(2^{2b+1}-16)^2+2^{2b+6}(2^{2b-4}+1). \\
\end{aligned}\right.
$$
Since $b \ge 4$, at least one of $2b-2$ or $2b-4$, say $2b-2$, has an odd prime factor. Hence $2^{2b-2}+1$ is not a prime and any prime dividing it is always congruent to $1$ modulo $4$.
Since the number of representations of the integer $2^{2b-2}+1$ by a sum of $2$ squares is equal to $4\sum_{d\mid 2^{2b-2}+1} \left( \frac {-1}d\right)$, there are integers $s,t$ different from $2^{b-1}$ such that
$$
2^{2b-2}+1=s^2+t^2 \ \text{and} \ n= x^2+(2^{2b+1}-4)^2+(2^{b+2}s)^2+(2^{b+2}t)^2.
$$
 Therefore $n$ is a sum of $4$ squares of integers in $S_{\rho}$. 

If $n \equiv 3+\epsilon^2 \Mod 8$ for $\epsilon=0$ or $2$, then there are odd integers $a,b$, and $c$ such that
$$
n=\epsilon^2+a^2+b^2+c^2.
$$
The theorem follows directly from this.   

Now, assume that $n \equiv 2 \Mod 4$. Then there are nonnegative integers $x_i,y_i,$ and $z_i$ for $i=1,2$ with $x_i \equiv y_i \equiv 1\Mod 2$ and $z_i \equiv 0\Mod 2$ such that 
$$
n=x_1^2+y_1^2+z_1^2+4=x_2^2+y_2^2+z_2^2+16.
$$
If $z_1$ or $z_2$ is not equal to $\rho$, we are done. If $z_1=z_2=2^{a+1}$, then we have 
$$
n=\left\lbrace\begin{aligned} 
&\ x_1^2+y_1^2+(2^{a+1}-2)^2+2^{a+3}, \\
&\ x_2^2+y_2^2+(2^{a+1}-4)^2+2^{a+4}. \\
\end{aligned}\right.
$$
 In this case, the theorem follows directly from the fact that either $a+3$ or $a+4$ is even and $2a+2>a+4$.  
 
If $n\equiv 4 \Mod 8$, then $n$ is a sum of $4$ squares of odd integers. 

Finally, assume that $n$ is divisible by $8$. Let $n=2^{2s}t$, where  $t \equiv 2 \Mod 4$ or $t \equiv 4 \Mod 8$.  Assume that $s \ge a+2$. Since $t$ is a sum of $4$ squares of integers, there are integers $x,y,z$, and $w$ such that 
$$
n=2^{2s}t=(2^sx)^2+(2^sy)^2+(2^sz)^2+(2^sw)^2.
$$
Assume that $s \le a-2$. Then there are integers $x,y,z$, and $w$ such that 
$$
t=x^2+y^2+z^2+w^2, \ \text{where $x,y,z,w \in S_{2^{a+1-s}}$.}
$$
Therefore we have 
$$
n=2^{2s}t=(2^sx)^2+(2^sy)^2+(2^sz)^2+(2^sw)^2, \ \text{where $2^sx,2^sy,2^sz,2^sw \in S_{2^{a+1}}$}.
$$
Finally, assume that $s=a+1,a$, or $a-1$. Let $\epsilon=a+1-s$. If $t$ is a sum of $4$ square of integers in $S_{2^{\epsilon}}$, then $n$ is a sum of $4$ squares of integers in $S_{\rho}$.  
Note that for any integer $t$ with $t\equiv 2\Mod{4}$,$\ t\equiv 4\Mod{8}$, we have $k_{2^\epsilon}(t)>4$ if and only if $$
t = \begin{cases}
\ 2,6,10,14,30,46 & \text{if \ $\epsilon=0$,}\\
\ 6,14,22,30 & \text{if \ $\epsilon=1$,}
\end{cases}
$$
and there is no integer $t$ such that $k_4(t)>4$. For any integer $n=t\cdot \left(\frac {\rho}{2^{\epsilon}}\right)^2$ contained in
\[
N(\rho)=\left\lbrace2\rho^2\!, 6\rho^2\!, 10\rho^2\!, 14\rho^2\!, 30\rho^2\!, 46\rho^2\!, 6\left(\frac{\rho}{2}\right)^2\!\!,  14\left(\frac{\rho}{2}\right)^2\!\!,  22\left(\frac{\rho}{2}\right)^2\!\!, 30\left(\frac{\rho}{2}\right)^2 \right\rbrace\text,
\] 
one may easily  compute that $k_{\rho}(n)=5$. This completes the proof.
\end{proof}

\begin{thm} Let $a$ be a nonnegative integer and let $\rho=3\cdot2^{a+1}$. Then any positive integer $n$ except for $14\cdot2^{2a+2}$ is a sum of $4$ squares of integers in $S_{\rho}$. Furthermore, $14\cdot2^{2a+2}$ is a sum of $5$ squares of integers in $S_{\rho}$ and hence $M(\rho)=5$. 
\end{thm}

\begin{proof} The method of the proof is quite similar to that of Theorem ~\ref{3.2}. One may consider $n-(3\cdot 2^\alpha)^2$ instead of $n-(2^\alpha)^2$ for $\alpha=1,2$, and $4$ in the previous theorem. Since all the other things are quite similar to the above theorem, the proof is left as an exercise to the reader.
\end{proof}

Now, we assume that $\rho=2^s\cdot3^t$, where $s\geq0$ and $t\geq2$. 

\begin{prop}\label{3pri}
   Assume that an integer $n$ is a sum of three squares  of integers. Then there exist integers $a,\ b$, and $c$ such that 
$$
n=a^2+b^2+c^2 \  \text{and}  \ \gcd(a,b,c,3)=1.
$$ 
In particular, if n is divisible by $3$, then $n$ is a sum of three squares of integers not divisible by $3$.  
\end{prop}

\begin{proof}
    Since  the ternary quadratic form $I_3=x^2+y^2+z^2$ is isotropic and unimodular over $\mathbb{Z}_3$, it is primitively universal over $\mathbb{Z}_3$, that is, it primitively represents all integers over $\mathbb Z_3$. Moreover, the class number of $I_3$ is $1$. So, if $n$ is represented by $I_3$ over $\mathbb{Z}$, then there exists a  representation of $n$ by $I_3$ over $\mathbb{Z}$ which is a primitive representation over $\mathbb Z_3$ by 102.5 of  \cite{OM}.
  Therefore if an integer $n$ divisible by $3$ is represented by $I_3$, then there are integers $a,b$, and $c$ such that 
   $$
   n=a^2+b^2+c^2 \ \text{and} \ abc \not \equiv 0 \Mod 3. 
   $$
   This completes the proof.
\end{proof}

\begin{prop}\label{2mod3}
    Let $a_0,\ a_1$, and $a_2$ be nonzero integers such that $a_0 \equiv 0 \Mod 9$, $a_1 \equiv 1 \Mod 3$, and $a_2\equiv 2\Mod 3$. Then
    there are integers $b_0,\ b_1$, and $b_2$ such that 
    $$
    b_0^2 + b_1^2 + b_2^2=a_0^2 + a_1^2 + a_2^2  \
      \text{and} 
    \ \pm a_0 \notin \{b_0, b_1, b_2\}.
    $$
\end{prop}

\begin{proof}
If we let $a_0 = 3A$, then $A\equiv 0\Mod3$. Clearly, both $m = (a_0 + a_1 + a_2)/3$ and $m' = (-a_0 + a_1 + a_2)/3 = m - 2A$ are integers. Note that
\begin{align*}
 a_0^2 + a_1^2 + a_2^2 & = (2m-a_0)^2 + (2m-a_1)^2 + (2m-a_2)^2,\\
  (-a_0)^2 + a_1^2 + a_2^2 & = (2m'+a_0)^2 + (2m'-a_1)^2 + (2m'-a_2)^2\text.
\end{align*}
Define $b_i = 2m - a_i$ for $0\le i\le 2$, and
\[
 b'_0 = 2m' + a_0 \text,\quad b'_1 = 2m' - a_1 \text,\quad b'_2 = 2m' - a_2\text.
\]
If $\pm a_0 \notin \{b_0, b_1, b_2\}$, then we are done. Suppose that $ua_0 = b_i$ for some $u\in \{\pm 1\}$ and $0\le i\le 2$. We claim that $\pm a_0 \notin \{b'_0, b'_1, b'_2\}$. Let $\{i, j, k\} = \{0, 1, 2\}$. Since $m \equiv m' \sMod3$, we have
\[
 b_0 \equiv b'_0 \not\equiv b_1 \equiv b'_1 \not\equiv b_2 \equiv b'_2 \not\equiv b_0 \Mod3\text.
\]
Hence neither $b'_j$ nor $b'_k$ is equal to $\pm a_0$. Now, suppose that $b'_i =\pm a_0$ for a suitable sign. Then we have
\[
 b'_i - b_i = \pm a_0 - ua_0 \in \{0, \pm 2a_0\} = \{0, \pm 6A\}\text.
\]
However, $b'_1 - b_1 = b'_2 - b_2 = 2m' - 2m = -4A$ and $b'_0 - b_0 = 2m' + a_0 - 2m + a_0 = 2A$, which is absurd. This proves the proposition.
\end{proof}

\begin{thm} \label{3.6}
   If $\rho$ is divisible by $9$, then $M(\rho)=4$. 
\end{thm}

\begin{proof}
    Let $n$ be any positive integer. Then there are integers $a,b,c$, and $d$ such that 
$$
n=a^2+b^2+c^2+d^2,
$$ 
by Lagrange's four square theorem.
   
 Assume that $n\equiv0\Mod3 $. Then all of $a,b,c$, and $d$ are divisible by $3$ or exactly one of them is divisible by $3$. If $\lvert a \rvert=\lvert b\rvert=\lvert c \rvert= \lvert d \rvert=\rho$, then $n=(2\rho)^2+0^2+0^2+0^2$. So we may suppose that at least one of them is not $\rho$. Without loss of generality, we assume that  $\lvert d\rvert\neq\rho$. Then $n-d^2\equiv 0,2 \Mod{3}$ is a sum of three squares. Hence, by Propositions~\ref{3pri} and ~\ref{2mod3}, $n-d^2$ is a sum of three squares of integers in $S_\rho$.
    
   Assume that  $n\equiv1\Mod3 $. Then we may assume $a\not\equiv0\Mod3 $. Since $n-a^2\equiv0\Mod3 $, $n-a^2$ is a sum of three squares not divisible by $3$ by Proposition~\ref{3pri}, which is desired.
    
    Finally, assume that  $n\equiv2\Mod3 $. Then exactly two of $a,b,c,$ and $d$ are divisible by $3$. Without loss of generality,
    we assume that $a,b\equiv0\Mod3$. If neither $\lvert a \rvert$ nor $\lvert b \rvert$ is equal to $\rho$, then we are done.
    If $\lvert a \rvert=\rho$ and $\lvert b \rvert\neq \rho$, then by Proposition~\ref{2mod3}, there are integers $a', c'$ and $d'$ such that  
    $$
    n-b^2=(a')^2+(c')^2+(d')^2,
    $$ 
    where all of $\lvert a' \rvert$, $\lvert c' \rvert$, and $\lvert d' \rvert$ are not equal to $\rho$. Assume that $\lvert a \rvert= \lvert b \rvert=\rho$. Then by Proposition~\ref{2mod3}, there are integers
    $a', c'$, and $d'$ such that 
    $$
    n-b^2=(a')^2+(c')^2+(d')^2,
    $$ 
    where all of $\lvert a' \rvert$, $\lvert c' \rvert$, and $\lvert d' \rvert$ are not equal to $\rho$. We may assume that $a'$ is divisible by $3$. Then by Proposition~\ref{2mod3} again, there are integers $b', c''$, and $d''$ such that
    $$
    n-(a')^2=(b')^2+(c'')^2+(d'')^2,
    $$
    where all of $\lvert b' \rvert, \lvert c'' \rvert$, and $\lvert d'' \rvert$  are not equal to $\rho$. The theorem follows directly from this.
\end{proof}

\end{document}